\documentclass[paper=letter, fontsize=10pt]{article}
\usepackage[english]{babel}
	
\usepackage{amsmath}									
\usepackage{amsfonts}
\usepackage{amsthm}
\usepackage{amssymb}

\usepackage{authblk}
\usepackage{enumitem}

\usepackage{eucal}

\usepackage[pdftex]{graphicx}														% Enable pdflatex
\usepackage{url}

\usepackage{mathrsfs}
\usepackage{bbm}

\newtheorem{theorem}{Theorem}

\newtheorem{lemma}{Lemma}
\newtheorem{proposition}{Proposition}
\newtheorem{corollary}{Corollary}

\theoremstyle{definition}
\newtheorem{definition}{Definition}

\theoremstyle{remark}
\newtheorem*{remark*}{Remark}

\numberwithin{equation}{section}		 
\numberwithin{proposition}{section}			 
\numberwithin{table}{section}				 
\numberwithin{definition}{section}
\numberwithin{theorem}{section}
\numberwithin{corollary}{section}
\numberwithin{exercise}{section}
\numberwithin{question}{section}

\title{
		\vspace{-1in} 	
		\usefont{OT1}{bch}{b}{n}
		\normalfont \normalsize \textsc{} \\ [25pt]
		\huge Exact discretization of harmonic tensors
}

\author{\normalfont \large 
 T. Chumley\footnote{\scriptsize Department of Mathematics and Statistics, Mount Holyoke College},
\  R. Feres\footnote{\scriptsize Department of Mathematics and Statistics, Washington University in St.~Louis},
\  M. Wallace\footnote{\scriptsize School of Medicine, Washington University in St.~Louis}
}

\begin{document}

\date{}

\maketitle

\begin{abstract}
 Furstenberg \cite{furstenberg} and Lyons and Sullivan \cite{lyonssullivan} have shown how to discretize harmonic functions on a Riemannian manifold $M$ whose Brownian motion satisfies a certain recurrence property called $\ast$-recurrence. We study analogues of this discretization for tensor fields which are harmonic in the sense of the covariant Laplacian. We show that, under certain restrictions on the holonomy of the connection, the lifted diffusion on the orthonormal frame bundle has the same $\ast$-recurrence property as the original Brownian motion. This observation leads us to introduce a technique we call scalarization which reduces the problem of discretization for a tensor field to that of ordinary harmonic functions.  
\end{abstract}

\section{Introduction}
\label{intro} 
Inspired by an idea of Furstenberg \cite{furstenberg}, Lyons and Sullivan \cite{lyonssullivan} introduced a discretization procedure for harmonic functions on a Riemannian manifold $M$. Phrased in probabilistic language, they showed that if the Brownian motion on $M$ satisfies a certain recurrence property, then it is equivalent to a random walk on a certain discrete subset $X$ of $M$. Since this seminal work, there has been an ongoing program to study the connection between the potential theory of the Laplacian on Riemannian manifolds and the potential theory of Markov chains on discrete subsets through the use of similar discretization techniques \cite{BL, bak, kaimanovich}. 

Along such lines, our purpose here is to show that the Furstenburg-Lyons-Sullivan idea can be extended to discretize tensor fields on $M$ which are harmonic in the sense of the covariant Laplacian. 

Following \cite{lyonssullivan}, we consider the case where $M$ is a covering of a compact Riemannian base manifold $B$. 
It is natural to frame our discretization in the language of groupoids of linear isometries. For this purpose we introduce some terminology: by a {\em linear isometry} on the Riemannian manifold $M$ over a pair $(x,y)\in M\times M$, we mean a linear map $\sigma:T_xM\rightarrow T_yM$ that is orthogonal relative to the Riemannian inner products at $x$ and $y$. The collection of all linear isometries constitutes a fiber bundle \[
(s,t):\mathcal{I}(M)\rightarrow M\times M
\] 
where each $\sigma\in \mathcal{I}(M)$ is a linear isometry above the pair $(x,y)=(s(\sigma), t(\sigma))$. Note that $\mathcal{I}(M)$ is a groupoid for which $M$ may be identified with the set of identities, $M=\mathcal{I}(M)_0$. The set $\mathcal{I}(M)_2$ of composable elements consists of the pairs $(\sigma, \eta)\in \mathcal{I}(M)\times \mathcal{I}(M)$ for which  $t(\sigma)=s(\eta)$, so that  the composition $\eta\circ \sigma$ is defined and belongs to $\mathcal{I}(M)$.
  
\begin{definition}
Let  $M$ be a Riemannian manifold and $X$ a subset of $M$. A {\em groupoid of linear isometries over $X$} is a subgroupoid of $\mathcal{I}(M)$ with $\mathcal{G}_0=X$. We set $\mathcal{G}_x=\{\sigma\in \mathcal{G}: s(\sigma)=x\}$, $\mathcal{G}^y=\{\sigma\in \mathcal{G}: t(\sigma)=y\}$, and $\mathcal{G}_x^y=\mathcal{G}_x\cap \mathcal{G}^y$.
  \end{definition}
  
\begin{definition}
A {\em random walk} on a countable groupoid $\mathcal{G}$ is a family $\mu=\{\mu_\sigma: \sigma\in \mathcal{G}\}$ where $\mu_\sigma$ is a probability measure on $\mathcal{G}_{t(\sigma)}$.
  We call $\mu$ the family of {\em transition probabilities}. 
\end{definition}
  
By a tensor field $\tau$ on a subset $X\subset M$, we mean simply a family of tensors $\tau_x$ on $T_xM$ for each $x\in M$. Thus, a tensor field on $M$ gives rise by restriction to a tensor field on any subset of $M$. We will define a notion of harmonicity for such restricted tensors by averaging over a random walk, as follows:
 
\begin{definition}\label{def:mu-harmonic}
Let $\mu$ define a random walk on a countable groupoid of linear isometries $\mathcal{G}$ over $X\subset M$. A tensor field $\tau$ on $X$ will be called {\em $\mu$-harmonic} if for each $x\in X$  
\[
\tau_x=\sum_{\sigma\in \mathcal{G}_x} \mu_x(\sigma) \sigma^{-1}\tau_{t(\sigma)}.
\]
Here, $\sigma^{-1}$ denotes the canonical extension of the isomorphism $\sigma^{-1}:T_{t(\sigma)}M\rightarrow T_{s(\sigma)}M$ to the corresponding tensor spaces. Also, recall that we identify $M$ with $\mathcal{I}(M)_0$, so that the $\mu_x=\mu_{\mathrm{id}_{T_xM}}$
\end{definition}

With these definitions, we can state our result as follows:

\begin{theorem}
Let $M$ be a covering space of a real analytic compact Riemannian manifold $B$. Let $X\subset M$ be the inverse image of a single point $x\in B$. Then there exists a countable groupoid  $\mathcal{G}$ of linear isometries over $X$ and a random walk $\mu$ on $\mathcal{G}$ that discretizes bounded harmonic tensor fields. In other words, if $\tau$ is a bounded tensor field on $M$ harmonic for the covariant Laplacian, then $\tau|_X$ is a tensor field on $X$ harmonic for $\mu$ in the sense of Definition \ref{def:mu-harmonic}. If $M$ is simply connected, $\mathcal{G}$ may be chosen  so that  $\#\mathcal{G}_x^y=1$  for all $(x,y)\in X\times X$. 
\end{theorem}
  
\begin{remark*} 
The hypothesis of real analyticity covers many cases of interest, but it is used here for simplifying the exposition. More explicitly, we make use of this hypothesis in establishing a condition on the infinitesimal holonomy Lie algebra. Our result is stated in a slightly more general form as Theorem \ref{thm:main-theorem} below. 
\end{remark*}

We should note that the tensor fields we discretize are harmonic for the covariant Laplacian, which in general is different from the Hodge-de Rham Laplacian typically used for differential forms. Discretizing forms that are harmonic for the latter operator would be of some interest as well. For example, the possibility of discretizing the harmonic 1-form $\overline{\partial}f$, where $f$ is a harmonic function on a Kahler manifold $M$, and using a converse to Theorem 1.1 as in \cite{kaimanovich}, would lead to a discrete $\overline{\partial}$ operator. We wonder in this regard whether an intrinsic and probabilistic characterization of the discrete holomorphic functions (those which are the restriction of holomorphic functions on $M$)  can be obtained without reference to the ambient Kahler manifold itself. We leave this line of investigation to a future work.
  
The rest of the paper is organized as follows. In Section \ref{basic-disc}, we review the discretization of harmonic functions given in \cite{lyonssullivan}. In summarizing their work, we have taken care to emphasize the probabilistic interpretation of their procedure. In Section \ref{sec:geom-prelim}, we introduce some preliminary background on geometry of the orthogonal frame bundle and set the stage for showing that the Lyons-Sullivan discretization can be lifted to this setting. We conclude in Section \ref{sec:discretization} with a proof of our main theorem using the machinery built in Section \ref{sec:geom-prelim}.
  
 \section{The basic discretization argument} 
\label{basic-disc}
We begin in this section by reviewing the Lyons-Sullivan discretization, paying close attention to the generalization we have in mind. Details about these results can be found in \cite{kaimanovich} or \cite{lyonssullivan}.

\subsection{$\ast$-recurrent sets}

Consider the following set-up: $N$ is a manifold, $\mathcal{A}$ is a second-order linear differential operator on $N$, and $(\Omega, \mathcal{F}_t, \xi_t, \mathbb{P}_x)$ a diffusion process corresponding to $\mathcal{A}$. The main examples we have in mind for $N$ and $\mathcal{A}$ are: 
 \begin{enumerate} 
\item $N=(M,g)$, a $d$-dimensional Riemannian manifold and $\mathcal{A}=\frac{1}{2}\Delta_M$, the Laplace-Beltrami operator. 
\item $N=(M,g)$ as above, but $\mathcal{A}=\frac{1}{2}\Delta_M+Z$ for a drift vector field $Z\in\mathcal{T}(M)$. 
\item $N=O(M)$, the bundle of orthonormal frames over $M$, and $\mathcal{A}=\sum_{i=1}^r H_i^2$ where $H_1,\ldots,H_d$ are the canonical horizontal vector fields on $O(M)$. (See below for definitions.)
\end{enumerate}
In each of these situations, a function $h:N\rightarrow \mathbb{R}$ is called \emph{harmonic} if $\mathcal{A}h=0$; our interest is in discretizing such functions as in \cite{kaimanovich, lyonssullivan}. Although the results in these papers are stated for the case of Brownian motion only, we can see by repeating proofs verbatim that the discretization works in any of our cases of interest, so long as $\xi$ escapes precompact sets in finite time and possesses the following recurrence property:  

\begin{definition}\label{starrec} A discrete subset $X$ of $N$ is called \emph{$\ast$-recurrent} if, for each $x\in X$, there exist $E_x, V_x\subset M$ with the following properties:  
\begin{enumerate}
\item $x\in E_{x}\subset V_{x}$;
\item $V_{x}$ is precompact and open;
\item the $E_{x}$'s are closed and pairwise disjoint;
\item $E:=\bigcup_{x\in X}E_{x}$ is closed and recurrent for $\xi$; 
\item $\sup_{x\in X}C(V_{x},E_{x})=:C<\infty$ where $C(V,E)$ is the \emph{Harnack constant,} defined for $E\subset\subset V$ as \[
C(V,E):=\sup_{h}\sup_{x,y\in E}\frac{h(x)}{h(y)}.
\] where the sup ranges over all positive harmonic functions on $N$. 
\end{enumerate}\end{definition}

The condition on Harnack constants has the following probabilistic interpretation: given $y\in E$, let $x(y)$ be the unique index $x\in X$ for which $y\in E_{x}$, and then let $\varepsilon_{y}^{V_{x}}$ be the exit distribution from $V_{x}$ of $\xi_t$ starting at $y$. That is,  \[
\varepsilon_{y}^{V_{x}}\left[A\right]=\mathbb{P}_{y}\left[\xi_{\tau}\in A\right]
\]
where $A$ is a Borel subset of $\partial V_{x}$ and
\begin{equation}\label{eq:tau}
\tau=\inf\left\{ t>0\,:\,\xi_{t}\not\in V_{x(\xi_{0})}\right\} 
\end{equation} is the first exit time from $V_x$. Then item 5 in Definition \ref{starrec} is equivalent to the condition that all the exit distributions are mutually absolutely continuous, with 
\[
\frac{1}{C}\leq\frac{\mathrm{d}\varepsilon_{x}^{V_{x}}}{\mathrm{d}\varepsilon_{y}^{V_{y}}}\leq C\quad\forall y\in E,\, x=x(y).
\]
Since we are assuming that $\xi$ leaves precompact open sets, $\tau$ is finite a.s. and $\varepsilon_x^V(\partial V)=1$ whenever $x\in V$. 

\subsection{The Lyons-Sullivan discretization}

Suppose now that we are given a $\ast$-recurrent set $X$ as above. Define sequences of stopping times $\{\tau_n\}_{n\geq 0}$ and $\{T_n\}_{n\geq 1}$ inductively as follows: 
\begin{alignat*}{1}
\tau_0 &:=  \begin{cases} 0 &\mbox{if } \xi_0\not\in X \\ \tau &\mbox{if } \xi_0\in X \end{cases} \\
T_{n} & :=\inf\left\{ t>\tau_{n-1}\,:\,\xi_{t}\in E\right\} =\tau_{n-1}+T_{E}\circ\theta_{\tau_{n-1}}\\
\tau_{n} & :=\inf\left\{ t>T_{n}\,:\,\xi_{t}\not\in V_{x(\xi(T_{n}))}\right\} =T_{n}+\tau\circ\theta_{T_{n}}
\end{alignat*}
Here, $\tau$ is still defined as at \eqref{eq:tau} and $T_A$ denotes the first hitting time of $\xi_t$ to a set $A$, i.e. 
\[ 
T_A=\inf\{t>0 : \xi_t\in A\}.
\] 
Note that our assumptions on $\xi$ --- namely, $\ast$-recurrence and leaving precompact sets --- ensure that $T_n$ and $\tau_n$ are all a.s. finite for all $n$. Finally, define three random sequences in $N$:
\[
X_{n}=x\left(\xi(T_{n})\right),\qquad Y_{n}=\xi(T_{n}),\qquad Z_{n}=\xi(\tau_{n})
\] 

The idea of the discretization is apply the acceptance-rejection algorithm to the random sequence of exit distributions and $\varepsilon_{X_{n}}^{V_{X_{n}}}$ and $\varepsilon_{Y_{n}}^{V_{X_{n}}}$. Accordingly, we enlarge the sample space $\Omega$ to accommodate a sequence of random variables $\left\{ U_{n}\right\}$ such that, under each $\mathbb{P}_x$, $\{U_n\}$ and $\xi$ are all independent, and $U_n$ has a uniform distribution on $[0,1]$. Put $N_0=0$ and 
\[
N_k  :=\min\left\{ n\geq N_{k-1} \,:\,  U_{n}\leq\frac{1}{C}\frac{\mathrm{d}\varepsilon_{X_{n}}^{V_{X_{n}}}}{\mathrm{d}\varepsilon_{Y_{n}}^{V_{X_{n}}}}\left(Z_{n}\right)  \right\} \] for $k\geq 1$. With this setup, the $\ast$-recurrence condition ensures that the numbers \[ \mu_a(x) := \mathbb{P}_a\left[ X_{N_1} = x\right] \qquad (x\in X)\] form a probability distribution on $X$. In fact, $\left\{ X_{N_k} \, : \, k\geq 1\right\}$ is a Markov chain on $X$ which discretizes harmonic functions on $N$ in the following sense: 

  \begin{proposition}[\cite{kaimanovich, lyonssullivan}]\label{prop:lyons-sullivan} The sequence $\left\{X_{N_k}\,:\,k\geq1\right\}$ is a Markov chain on $X$ with transition probabilities \[p(x,y)=\mu_x(y)\qquad (x,y\in X).\] Furthermore, if $h$ is a bounded harmonic function on $N$, then $h|_X$ is a bounded harmonic function on $X$ in the sense that \[ h(x) = \sum_{y\in X} p(x,y)h(y).\] \end{proposition}

\begin{remark*}
In \cite{kaimanovich}, Kaimanovich shows that a converse holds under an additional hypothesis called the \emph{uniform core condition}. Namely, when this condition holds, any bounded function on $X$ which is harmonic with respect to random walk defined by $p(x,y)$ can be extended to a function on $M$ which is harmonic with respect to the Brownian motion. Thus, in this case, bounded harmonic functions on $M$ and bounded harmonic functions on $X$ are in one-to-one correspondence. 
\end{remark*}

\section{Geometric preliminaries} \label{sec:geom-prelim}

In this section, we gather some geometric preliminaries required for our discretization of harmonic tensors. To motivate, we briefly summarize the strategy:
\begin{enumerate}
\item Given a $\ast$-recurrent subset $X$ of $ M$, show that, under suitable hypotheses, there exists a corresponding $\ast$-recurrent subset $\widetilde{X}$ of $O(M)$, or of a suitable reduction $P$ of $O(M)$. 
\item Given a harmonic tensor field $\tau$ on $M$, pass to the \emph{scalarization} $F_\tau$; this is a tensor-valued map on $O(M)$ (or on $P$) whose component functions are harmonic for the \emph{horizontal Laplacian} defined below. 
\item Discretize $F_\tau$ according to the Lyons-Sullivan method of Section \ref{basic-disc}.
\item Translate the result into the language of random walks on groupoids of isometries.
\end{enumerate}
The lifting of $X$ to $\widetilde{X}$ in Step 1 requires controlling the corresponding Harnack constants in the definition of $\ast$-recurrence (cf. Item 5 in Definition \ref{starrec}). For this purpose, we recall in Section \ref{hypo-hormander} some facts about hypoellipticity and H\"{o}rmander vector fields, and in Section \ref{holo} some additional facts about the holonomy reduction. Using this material, we show in Section \ref{star-rec-subsets} the possibility of constructing $\ast$-recurrent subsets of $O(M)$ and $M$ with the required properties, under a certain hypothesis on the infinitesimal holonomy Lie algebra. (This assumption is likely not optimal, but we adopt it here for simplicity.) 

For the convenience of the reader, we begin by recalling the relationships between frame bundles, scalarizations and the rough Laplacian in Section \ref{sec:scalarization}. References for this material include \cite{bleecker, hsu, kn}. We defer Steps 3 and 4 to Section \ref{sec:discretization}.

\subsection{Frame bundles, scalarization and the rough Laplacian}  
\label{sec:scalarization}

Let $M$ be a $d$-dimensional smooth manifold. A \emph{linear frame at $x$} is a linear isomorphism $\sigma:\mathbb{R}^d\rightarrow T_x M$. The collection of all such frames is written $GL(M)$ and called the \emph{frame bundle} over $M$; it is a principal fiber bundle with structure group $GL(d,\mathbb{R})$ and a natural projection $\pi:GL(M)\rightarrow M$ which maps the frame $\sigma:\mathbb{R}^d\rightarrow T_x M$ in $GL(M)$ to the point $x$ in $M$. Similarly, if $M$ has a Riemannian metric $g$, then an \emph{orthonormal frame at $x$} is an isometry $\sigma:(\mathbb{R}^d,\overline{g})\rightarrow (T_x M, g_x)$, where $\overline{g}$ is the Euclidean metric on $\mathbb{R}^d$. The collection of all such maps is written $O(M)$ and called the \emph{orthonormal frame bundle} over $M$; it is a reduction of $GL(M)$ whose structure group is $O(d,\mathbb{R})$.

One reason for introducing $GL(M)$ and $O(M)$ is that tensor fields on $M$ can be lifted to tensor-valued maps on $GL(M)$ or $O(M)$. This procedure is called \emph{scalarization.} To define it, let $\mathcal{T}^r_s(M)$ be the space of $(r,s)$-tensor fields on $M$ (meaning $r$ upper indices and $s$ lower indices), and let $T^r_s\left(\mathbb{R}^d\right)$ be the tensor product 
\[
T_r^s\left(\mathbb{R}^d\right) = \underbrace{\mathbb{R}^d\otimes\cdots\otimes\mathbb{R}^d}_{r}\otimes\underbrace{\left(\mathbb{R}^d\right)^\ast\cdots\otimes \left(\mathbb{R}^d\right)^\ast}_{s}.
\] 
Equivalently, $T^r_s\left(\mathbb{R}^d\right)$ is the space of all multi-linear mappings 
\[
F:\underbrace{\left(\mathbb{R}^d\right)^\ast \times\cdots\times\left(\mathbb{R}^d\right)^\ast}_{r} \times \underbrace{\mathbb{R}^d\cdots\times \mathbb{R}^d}_{s}\rightarrow\mathbb{R}
\]  In terms of the latter definition we have a natural action of $GL(d,\mathbb{R})$ on $T^r_s\left(\mathbb{R}^d\right)$, namely 
\begin{equation}\label{eq:action}
gF(\omega^1,\ldots,\omega^r,v_1,\ldots,v_s) = F(\omega^1\circ g^{-1},\ldots,\omega^r\circ g^{-1},gv_1,\ldots, gv_s).
\end{equation}  for $g\in GL(d,\mathbb{R})$, $\omega^i\in \left(\mathbb{R}^d\right)^\ast$ and $v_j\in\mathbb{R}^d$. With these notations, we have: 

\begin{definition}[Scalarization of a tensor field]\label{def:scalarization}
Given $\tau\in\mathcal{T}^s_r(M)$, define the $T^r_s(\mathbb{R}^d)$-valued function $F_\tau$ on $GL(M)$ by the equivalention 
\[
F_\tau(\sigma):=\sigma^{-1} \tau(\pi(\sigma)).
\]
where $\sigma^{-1}:T^r_s M_x\rightarrow T^r_s\mathbb{R}^d$ indicates the isomorphism of the tensor algebras induced by the isomorphism $\sigma^{-1}:T_xM\rightarrow \mathbb{R}^d$ (cf. \cite{kn} Proposition I.2.12). The function $F_\tau$ is called the {\em scalarization} of $\tau$; it is equivariant in the sense that  
\[
F_\tau(\sigma g)=g^{-1}F_\tau(\sigma)
\]
for all $\sigma\in P$ and $g\in GL(d,\mathbb{R})$, where $g^{-1}$ acts on $F_\tau(\sigma)$ as at \eqref{eq:action}. \end{definition} 

\begin{remark*} 
If $f\in C^\infty(M)$ is a smooth function, i.e. a tensor field of type $(0,0)$, then the scalarization $F_f$ is simply the lift of $f$ to $O(M)$, i.e. $F_f(\sigma)=f(\pi(\sigma))$. 
\end{remark*}

Another reason for introducing frame bundles is the presence of invariantly defined \emph{standard horizontal vector fields}. To define these, we require that $O(M)$ be equipped with a \emph{connection,} i.e. a splitting of the tangent bundle $T O(M)$ as direct sum $TO(M)=H\oplus V$, where $H$ is the \emph{horizontal} subbundle defined by the connection, and $V$ is the \emph{vertical} subbundle defined by \[ 
V_\sigma = \{ Y\in T_\sigma O(M) : \pi_\ast Y = 0\}.
\]
Given a vector $X\in T_xM$, we write $X^h_\sigma$ for its \emph{horizontal lift} of $X$ to $H_\sigma$, where $\pi(\sigma)=x$. Precisely, $X^h_\sigma$ is the unique vector in $H_\sigma$ such that $\pi_\ast X^h_\sigma=X$. The horizontal lift of a vector field $X\in \mathcal{T}(M)$ is defined similarly. (See Section \ref{holo} below, or \cite{bc64,kn}, for more information on linear connections.) With this setup, we can define: 

\begin{definition}[Standard horizontal vector fields]\label{def:canonical}
Given $v\in \mathbb{R}^d$, define the horizontal vector field $H(v)$ on $O(M)$ by \[ H(v)_\sigma = \sigma(v)^h_\sigma\quad\mbox{ for } \sigma\in O(M).\] That is, $H(v)_\sigma$ is the horizontal lift to $T_\sigma O(M)$ of $\sigma(v)\in T_xM$. We write $H_i=H(e_i)$ where $e_i$ is the standard basis of $\mathbb{R}^d$ and call $H_1,\ldots, H_d$ the \emph{standard horizontal vector fields} on $O(M)$. 
 \end{definition}

With the notation of Definition \ref{def:canonical}, we define the \emph{horizontal Laplacian} on $O(M)$ by 
\begin{equation}
\label{eq:horizontal-laplacian}
\Delta_{O(M)} = \sum_{i=1}^d H_iH_i.
\end{equation} 
The horizontal Laplacian associated with the Levi-Civita connection is the lift of the Laplace-Beltrami $\Delta_M$ in the sense that (cf. \cite{hsu} Proposition 3.1.2)  \[ 
 (\Delta_M f)\circ \pi = \Delta_{O(M)}( f\circ\pi).
\] In view of the remark below Definition \ref{def:scalarization}, this last equation says that
\[
F_{\Delta_{M}f} = \Delta_{O(M)} F_f.
\]
In fact, the same is true if we replace the function $f$ by a tensor field $\tau\in\mathcal{T}^r_s(M)$ and $\Delta_Mf$ by the \emph{covariant Laplacian} $\Delta_M \tau$ defined by contracting (with respect to the metric) the two new indices $\nabla\nabla\tau \in \mathcal{T}^r_{s+2}(M)$. We summarize these facts in the following (cf. \cite{hsu} p. 193): 

\begin{proposition}\label{scalLap}
Let $\tau$ be a tensor field on $M$ and $\Delta_M \tau$ the covariant Laplacian of $\tau$. Then  \[
F_{\Delta_M\tau}= \Delta_{O(M)} F_\tau
\]
where on the right-hand side, $\Delta_{O(M)}$ is understood to act on each component. 
\end{proposition}

\subsection{Hypoellipticity and H\"{o}rmander vector fields}
\label{hypo-hormander}
  
Let $N$ be a smooth manifold of dimension $d$ and  $A_1, \dots, A_r$ smooth vector fields on $N$. Then $\{A_i\}_{i=1}^r$ is said to satisfy \emph{H\"ormander's bracket condition} if

\begin{equation}
\label{eq:hormandercond}
\mathrm{Lie}_x(A_1,\ldots,A_r)=T_x N \mbox{ for all }  x\in N.
\end{equation}
Here, $\mathrm{Lie}_x(\{A_i\}_{i=1}^r)$ denotes the values at $x$ of the Lie algebra generated by $\{A_i\}_{i=1}^r$, i.e. the linear span of \[ \{A_i(x)\}_{i=1}^r \cup \{[A_i, A_j](x)\}_{i,j=1}^r \cup \{[A_i,[A_j,A_k]](x)\}_{i,j,k=1}^r\cup\cdots .\]
The importance of \eqref{eq:hormandercond} is that a celebrated theorem of H\"ormander (cf. \cite{hormander}) asserts the \emph{hypoellipticity} of the sum-of-squares operator \[ L=\sum_{i=1}^r A_iA_i\] when \eqref{eq:hormandercond} holds. In PDE theory, $L$ is called \emph{hypoelliptic} if $L\varphi =f$ (in the sense of distributions) implies that $\varphi$ is smooth wherever $f$ is. A probabilistic consequence of hypoellipticity is the following: let $\xi=\xi(t,x,w)$ be the solution to the Stratonovich SDE 
\[
\xi_t = \xi_0 + \int_0^t A_i\left(\xi_s\right)\circ dw^i(s)
\] 
with the initial condition $\xi_0=x$, where $w(t)=(w^1(t),\ldots,w^r(t))$ is the coordinate process on the path space $W_0^r(\mathbb{R}^r)$, made into an $r$-dimensional Brownian motion by the Weiner measure  $P^W$. According to Theorem V.1.2 of \cite{iw}, $\xi$ is the diffusion process associated with $L$ in the sense of Definition V.1.1, \emph{ibid.} In this situation, hypoellipticity of $L$ entails that the image measure of $P^W$ on $M$ induced by the mapping $w\mapsto \xi(t,x,w)$ is absolutely continuous with respect to a fixed Riemannian metric on $M$. In fact, more is true: 

\begin{proposition} Let $\mu$ be a smooth, nonvanishing measure on $M$. Then $\xi_t$ has a smooth, positive density $p_t(x,y)$ with respect to $\mu$, for all $t>0$.  \end{proposition}

In fact, one can prove Gaussian estimates for $p_t(x,y)$. See \cite{bramanti} and references therein for more information on H\"ormander's vector fields. In particular, see \cite{jerison} for Gaussian estimates in the smooth manifold case. It will also be important for us to know that Harnack inequalities hold. The following is a special case of Theorem 7.2 of \cite{bony}.
 \begin{theorem}[Harnack's inequality]\label{harnack}
Suppose the vector fields $\{A_i\}_{i=1}^r$ on $N$ satisfy H\"ormander's condition \eqref{eq:hormandercond}. Then for any open $U\subset N$ and compact $K\subset U$, there exists a constant $c=c(K,U)$, with the property that $\sup_{x\in K} u(x) \leq c u(y)$ whenever $f$ is a positive function in $U$ satisfying $Lf=0$. 
 \end{theorem}
      
\subsection{Infinitesimal holonomy and hypoellepticity of the horizontal Laplacian}
\label{holo}

We now discuss some conditions which ensure that the standard horizontal vector fields of Definition \ref{def:canonical} satisfy H\"ormander's condition \eqref{eq:hormandercond}, so that the horizontal Laplacian defined at \eqref{eq:horizontal-laplacian} is hypoelliptic. Continue to let $M$ be a smooth connected manifold of dimension $d\geq 2$ with Riemannian metric $g$ and linear connection $H$ compatible with $g$. By \cite{kn} Proposition III.1.5, compatibility implies that the given connection is reducible to one on the orthonormal frame bundle $O(M)$ over $M$. 

In fact, by \cite{kn} Theorem II.7.2 (the reduction theorem) the given connection is reducible to one on $P(M,G)$, the holonomy reduction of $O(M)$. Here, $P$ is the holonomy sub-bundle through a fixed reference point $\sigma_0\in O(M)$ and $G$ is the corresponding holonomy group. Let $G_0$ be the restricted holonomy group, i.e. the component of the identity of $G$. Note that $P$ is connected, that $G_0$ is a normal Lie subgroup of $G$ and that $G/G_0$ is countable (cf. \cite{kn} Theorem II.4.2). The Lie algebra of $G$ (and of $G_0$) will be denoted $\mathfrak{g}$.

We write $\omega$ for the connection $1$-form associated with $H$, and $\Omega$ for the curvature. Recall that these are $\mathfrak{g}$-valued forms related by the \emph{structure equation} 
\begin{equation}\label{eq:structure}
\Omega = d\omega + \omega\wedge\omega.
\end{equation}
Recall also that $\omega$ is vertical (i.e., vanishes if its argument is horizontal) and $\Omega$ is horizontal (i.e., vanishes if any of its arguments is vertical). Thus, $\omega(X_u)$ is a measure of the vertical part of $X_u\in T_u P$. 
    
We now define the {\em infinitesimal holonomy Lie algebra} $\mathfrak{g}'(\sigma)$, following \cite{kn} page 96. Let $\sigma\in P$ and $\mathfrak{m}_1(\sigma)$ be the subspace of $\mathfrak{g}$ spanned by all elements of the form $\Omega_\sigma(X,Y)$, where $X, Y$ are horizontal vectors at $\sigma$. Inductively, define $\mathfrak{m}_k(\sigma)$, $k\geq 2$, to be the subspace of $\mathfrak{g}$ spanned by $\mathfrak{m}_{k-1}(\sigma)$ and by the values at $\sigma$ of all functions $f:P\rightarrow \mathfrak{g}$ of the form 
\[
f=U_k\dots U_1 \Omega(X,Y)
\] 
where $U_1,\ldots,U_k,X,Y$ are arbitrary horizontal vector fields on $P$. Finally, define $\mathfrak{g}'(\sigma)$ as  the union of the $\mathfrak{m}_k(\sigma)$ for $k\geq 1$. 

Proposition 10.4 of \cite{kn} shows that $\mathfrak{g}'(\sigma)$ is indeed a Lie subalgebra of $\mathfrak{g}$. The connected Lie subgroup of $G$ having the Lie algebra $\mathfrak{g}'(\sigma)$ is (by definition) the {\em infinitesimal holonomy group} at $\sigma$. In fact, $\mathfrak{g}'(\sigma)$ is a Lie subalgebra of $\mathfrak{g}^\ast(\sigma)$, the \emph{local holonomy Lie algebra} at $\sigma$, which in turn is Lie subalgebra of $\mathfrak{g}$:
\[ \mathfrak{g}'(\sigma)\subset\mathfrak{g}^\ast(\sigma)\subset \mathfrak{g} \]  
The following proposition, which summarizes Corollary II.10.7, Theorem II.10.8 and Proposition II.4.1 (b) of \cite{kn}, describes a condition which ensures that these Lie algebras coincide:
 
 \begin{proposition}\label{propoinf}
 If $\mathfrak{g}'(\sigma)$ has constant dimension on $P$, the infinitesimal, local and restricted holonomy Lie algebras (and their corresponding Lie groups) all coincide. This is the case, for example, when the connection is real analytic. 
\end{proposition} 

An alternate characterization of $\mathfrak{g}'(\sigma)$ will be more useful here. To state it, let $\mathcal{H}$ be the space of horizontal vector fields on $M$ (i.e. sections of $H$). Then let $\mathfrak{h}$ be the Lie algebra generated by $\mathcal{H}$, and $\mathfrak{h}(\sigma)$ the values at $\sigma$ of vector fields in $\mathfrak{h}$, i.e. 
\[
\mathfrak{h}(\sigma) = \{ A_\sigma \, : \, A\in \mathfrak{h} \}.
\]
With this setup, we have the following:

\begin{proposition}[\cite{omori}]
\label{verticalpart}
The infinitesimal holonomy Lie algebra of the connection at any $\sigma\in P$ is equal to the vertical part of $\mathfrak{h}(\sigma)$ in the sense that 
\[
\mathfrak{g}'(\sigma) = \omega_\sigma \left( \mathfrak{h}(\sigma) \right)\subset \mathfrak{g}.
\] 
\end{proposition}  

In \cite{omori}, this result is claimed to follow from p. 96 of \cite{kn}. However, we did not understand how the Proposition follows from the cited text, and therefore feel more comfortable supplying a proof. 

\begin{lemma}\label{lemmaVer}
Let $X,Y,U_1,\ldots,U_k\in\mathcal{H}$. Then
\begin{equation}\label{identityOmega}U_k \dots U_1\Omega(X,Y)=-\omega\left(v\left[U_k, v\left[ U_{k-1}, \dots,  v\left[ X,Y \right] \dots  \right]     \right] \right). \end{equation}
 \end{lemma}
 
\begin{proof}
 If $k=0$,  we have \[
 \Omega(X,Y)=-\omega([X,Y])=-\omega(v[X,Y])
 \]
by Corollary II.5.3 in \cite{kn} and the fact that $\omega$ is a vertical form.\footnote{Kobayashi and Nomizu follow a different convention regarding alternating tensors. For this reason, their formula contains an extra factor of $2$.} 

For the general case let us  write  $V:= v\left[ U_{k-1}, \dots,  v\left[ X,Y \right] \dots  \right]$. Then by induction
\[
U_k \dots U_1\Omega(X,Y)=-U_k\omega(V) = - (\mathcal{L}_{U_k}\omega)(V)-\omega([U_k,V]).\]
However, it follows from Cartan's formula and the structure equation \eqref{eq:structure} that 
\[
\mathcal{L}_U\omega(V) = d(i_U\omega)(V) + i_U(d\omega)(V) = d\omega(U,V) =0
 \] whenever $U\in\mathcal{H}$ and $V\in\mathcal{V}$.  Therefore 
 \[
 U_k \dots U_1\Omega(X,Y)= -\omega([U_k,V]) = -\omega(v[U_k,V])
 \]
and the result is true for all $k\geq 0$ by induction. 
 \end{proof}
 
By iterating the identity $[H,vX ]=[H,X]-[H,hX]$  we can express the left-hand side of Equation \eqref{identityOmega} in Lemma \ref{lemmaVer} as a linear combination of terms  $\omega([H_j', [H_{j-1}', \dots,[H_1',H_0']\dots]])$ for $j=1, \dots, k$, where the $H_i'$ are horizontal fields. 
 
\begin{proof}[Proof of Proposition \ref{verticalpart}]
The above remarks based on Lemma \ref{lemmaVer} show that the set of values at $\sigma$ of functions of the form 
 $f=U_k \dots U_1\Omega(X,Y)$ spans a linear subspace contained in $\omega_\sigma(\mathfrak{h}(\sigma))\subset \mathfrak{g}$. Recall that the infinitesimal holonomy Lie algebra $\mathfrak{g}'(\sigma)$ at $\sigma$ was defined earlier in terms of the set of values of functions of type $f$. Thus $\mathfrak{g}'(\sigma)$ is contained in $\omega_\sigma(\mathfrak{h}(\sigma))$. We need to show the opposite inclusion. This is done by induction on the number of brackets in an expression of the form
$[U_k, [U_{k-1}, \dots, [X, Y]\dots]].$ 

When $k=0$, $\omega_\sigma([X,Y])=-\Omega_\sigma(X,Y)\in \mathfrak{g}'(\sigma)$
and when $k=1$, 
$$\omega_\sigma([H_1, [X,Y]])= \omega_\sigma([H_1, h[X,Y]])+\omega_\sigma(H_1, v[X,Y]).$$
The first term on the right-hand side equals $-\Omega_\sigma(H_1, h[X,Y])$ while the second term, due to Lemma \ref{lemmaVer}, 
equals $-H_1\Omega(X,Y)$ at $\sigma$. Thus the desired conclusion holds for $k=0,1$.  Assuming it holds for $k-1$ for $k\geq 2$, then  by the   argument just used for $k-2$, and denoting $X:=[H_{k-1}, \dots, [X,Y]\dots]$,
$$ \omega_\sigma([H_k, X])=-\Omega_\sigma(H_k, hX)+ (H_1)_\sigma \omega(X).$$
The induction hypothesis applies to  the second term on the right-hand side to conclude the proof.
  \end{proof}

We apply these propositions as follows. In view of Definition \ref{def:canonical}, we can express $\mathcal{H}$ as 
\[ 
\mathcal{H} = \left\{  \sum_{i=1}^d f^i H_i \, :\, f^i \in C^\infty(M) \right\}
\]
and $\mathfrak{h}(\sigma)$ as the linear span of 
\[
\{H_i(\sigma)\}\cup \{[H_i,H_j](\sigma)\}\cup\{ [H_i,[H_j,H_k]](\sigma)\}\cup\cdots .
\]
Evidently, if $\mathfrak{g}'(\sigma)$ has constant dimension) holds as hypothesized in Proposition \ref{propoinf}, then Proposition \ref{verticalpart} implies that $\{H_i\}_{i=1}^d$ satisfies the H\"{o}rmander bracket condition at each $\sigma\in P$. Therefore the corresponding sum of squares operator
\[
\Delta_P := \sum_{i=1}^d H_i H_i
\] 
is hypoelliptic and enjoys all the properties of Section \ref{hypo-hormander}.

\subsection{$\ast$-recurrent subsets of the frame bundle}
\label{star-rec-subsets}
We are now prepared to construct a $\ast$-recurrent subset of $P$. In the interest of simplicity we shall henceforth assume the following:
\begin{enumerate}[label={(A{\arabic*})}]

\item $M$ is a Riemannian covering space of a compact, connected Riemannian manifold $B$ and $X$ is the inverse image in $M$ of a single point $x\in B$. 
\item $O(M)$ is equipped with the Levi-Civita connection induced by the metric $g$. 
\item The infinitesimal holonomy Lie algebra $\mathfrak{g}'(\sigma)$ has constant dimension, as in Proposition \ref{propoinf}. 
\end{enumerate}

In this case, the operator $L=\frac{1}{2}\Delta_P=\frac{1}{2}\sum_{i=1}^d H_iH_i$ is the horizontal lift of $A=\frac{1}{2}\Delta_M$, the Laplace operator. Furthermore, the set $X$ is manifestly $\ast$-recurrent for the diffusion on $M$ with generator $A$; the properties of Riemannian covering maps ensure that the condition in Definition \ref{starrec} on Harnack constants for $A$ is satisfied. 

We now study the same properties for the diffusion on $O(M)$ associated with $H$. In particular, we seek information about the Harnack constants for $L$. First we show that $L$ commutes with the group action in the following sense: 
 
\begin{proposition}\label{Lharm} With the notations above, \[ L(F\circ R_g)= (LF)\circ R_g\] for any $F\in C^\infty(P)$ and $g\in G$. 
\end{proposition}
 \begin{proof}
 Using $G$-invariance of the horizontal distribution and that $\pi\circ R_g = \pi$, it is easily shown that 
 $(d R_g)_\sigma H(e_i) = H(g^{-1}e_i)$ from which it can be shown to follow that 
 $$H(u)(f\circ R_g)= \left[H(g^{-1}u) f\right]\circ R_g$$ for any smooth function $f$ on $P$.
 It follows that
\begin{align*} \sum_i H(e_i) H(e_i) (f\circ R_g) &= \sum_i \left[H(g^{-1} e_i) H(g^{-1} e_i ) f\right]\circ R_g\\
&=\sum_{i,j,k} (g^{-1})_{ki} (g^{-1})_{ji} \left[H(e_k) H(e_j) f\right]\circ R_g\\
&=\sum_{j}  \left[H(e_j) H(e_j) f\right]\circ R_g.
\end{align*}
The last step is due to  $g$ being orthogonal.
 \end{proof}

 \begin{corollary}
The $L$-Harnack constant satisfies $C(E,V)=C(R_gE,R_gV)$, for all $g\in G$. 
 \end{corollary}
 
\begin{proof}
This follows from Proposition \ref{Lharm}, which implies that  $R_g$ maps bijectively   the positive  $L$-harmonic functions on  $V$ and those same functions on $R_g V$. 
  \end{proof}

Using this last Corollary, we can lift $X$ to a subset of the frame bundle which is still $\ast$-recurrent and has the same Harnack constants: 

\begin{proposition} \label{prop:star-rec-bundle} With assumptions (A1)-(A3) in force, let $P$ be the holonomy reduction of $O(M)$ through a fixed point $\sigma_0\in O(M)$. Let $\widetilde{X}\subset P$ be a set containing one representative point for each element in $P/G_0$ that maps into $X$ under the projection $P\rightarrow M$. Then $\widetilde{X}$ is a $*$-recurrent set for the diffusion process on $P$ associated with the operator $L$. 
\end{proposition}

Here, $G_0$ is the restricted holonomy group, which coincides with the infinitesimal holonomy group by (A3).
 
\begin{proof}
The main  remark is that under (A1) the horizontal diffusion operator $L$ is hypoelliptic. Now let $V$ be an open, connected subset of $N$ and 
$E$ a closed connected subset of $V$ containing $x$ in its interior.  The   inverse image components  of $E$ and $V$ under the covering map
will give a family of sets $E_{x'}, V_{x'}$ parametrized by the points $x'\in X$. Note that   $E_{x'}, V_{x'}$ are isometric to $E, V$, respectively,
so each pair has the same $\mathcal{A}$-Harnack constant as $E, V$. Now consider for each $\tilde{x}$ in $\widetilde{X}$ above
a given  $x'\in X$ the connected components of $\pi^{-1}(E_{x'})$ and $\pi^{-1}(V_{x'})$ containing $\tilde{x}$.
We call these sets $E_{\tilde{x}}$ and $V_{\tilde{x}}$. 
 Then (A2) will ensure
recurrence in the definition of $*$-recurrent set and (A1) guarantees hypoellipticity, which gives a Harnack inequality for $L$. 
The $\nabla$-affine isometry from $V_{x'}$ to $V$ given by the Riemannian covering space lifts to a diffeomorphism from $\pi^{-1}(V_{x'})$ to $\pi^{-1}(V)$ which commutes with $L$. It follows from this that the $L$-Harnack constant for every pair $E_{\tilde{x}}$, $V_{\tilde{x}}$ is the same. This gives that $\tilde{X}$ is $*$-recurrent.
\end{proof}
 
\section{Discretization of harmonic tensors}
\label{sec:discretization}  

We have now assembled all the pieces necessary for discretizing harmonic tensor fields; the rest is book keeping. In view of the discussion above concerning the holonomy Lie algebras, we can state our result in a slightly more general form than in Section \ref{intro}.

\begin{theorem}
\label{thm:main-theorem}
Let $M$ be a covering space of a real analytic compact Riemannian manifold $B$. Let $X\subset M$ be the inverse image of a single point $x\in B$. Then there exists a countable groupoid  $\mathcal{G}$ of linear isometries over $X$ that discretizes bounded harmonic tensors. If $M$ is simply connected, $\mathcal{G}$ may be chosen  so that  $\#\mathcal{G}_x^y=1$  for all $(x,y)\in X\times X$. 
\end{theorem}
  
\begin{proof}
Let $\widetilde{X}$ be the subset of $P$ from Proposition \ref{prop:star-rec-bundle}. According to that proposition, $\widetilde{X}$ is $\ast$-recurrent for the horizontal Brownian motion on $P$. 

Now let $\tau\in\mathcal{T}^r_s(M)$ be a tensor field on $M$ which is harmonic for the covariant Laplacian. Then the scalarization $F_\tau$ is harmonic for the horizontal Laplacian on $P$. By Proposition \ref{prop:lyons-sullivan}, we can discretize $F_\tau$, writing 

\[ F_\tau(u) = \sum_{v\in \tilde{X}} p(u,v)F_\tau (v).
\] 
Unpacking the definition of $F_\tau$, this means that
\[ 
\tau_{\pi(u)} = \sum_{v\in \tilde{X}} p(u,v) u\circ v^{-1}\tau_{\pi(v)}  = \sum_{v\in \tilde{X}} p(u,v) (v\circ u^{-1})^{-1}\tau_{\pi(v)}.
\] 
In this sum, $v\circ u^{-1}$ is an isometry $T_{\pi(u)}M\rightarrow T_{\pi(v)} M$. Therefore we can re-label the addends in terms of the base points and $\mathcal{G}_x$, so that 
\[
\tau_{x} = \sum_{\sigma\in\mathcal{G}_x} \mu_x(\sigma)\sigma^{-1}\tau_{t(\sigma)}
\] 
as claimed. 
\end{proof}


\begin{thebibliography}{10}

\bibitem{BL}
W.~Ballmann and F.~Ledrappier.
\newblock Discretization of positive harmonic functions on {R}iemannian
  manifolds and {M}artin boundary.
\newblock In {\em Actes de la {T}able {R}onde de {G}\'eom\'etrie
  {D}iff\'erentielle ({L}uminy, 1992)}, volume~1 of {\em S\'emin. Congr.},
  pages 77--92. Soc. Math. France, Paris, 1996.

\bibitem{bak}
W.~B\c{a}k.
\newblock A modification of the {L}yons-{S}ullivan discretization of positive
  harmonic functions.
\newblock {\em Potential Anal.}, 32(1):17--27, 2010.

\bibitem{bc64}
R.~L. Bishop and R.~J. Crittenden.
\newblock {\em Geometry of manifolds}.
\newblock Pure and Applied Mathematics, Vol. XV. Academic Press, New
  York-London, 1964.

\bibitem{bleecker}
D.~Bleecker.
\newblock {\em Gauge theory and variational principles}, volume~1 of {\em
  Global Analysis Pure and Applied Series A}.
\newblock Addison-Wesley Publishing Co., Reading, Mass., 1981.

\bibitem{bony}
J.-M. Bony.
\newblock Principe du maximum, in\'egalite de {H}arnack et unicit\'e du
  probl\`eme de {C}auchy pour les op\'erateurs elliptiques d\'eg\'en\'er\'es.
\newblock {\em Ann. Inst. Fourier (Grenoble)}, 19(fasc. 1):277--304 xii, 1969.

\bibitem{bramanti}
M.~Bramanti.
\newblock {\em An invitation to hypoelliptic operators and {H}\"ormander's
  vector fields}.
\newblock Springer Briefs in Mathematics. Springer, Cham, 2014.

\bibitem{furstenberg}
H.~Furstenberg.
\newblock Random walks and discrete subgroups of {L}ie groups.
\newblock In {\em Advances in {P}robability and {R}elated {T}opics, {V}ol. 1},
  pages 1--63. Dekker, New York, 1971.

\bibitem{hormander}
L.~H{\"o}rmander.
\newblock Hypoelliptic second order differential equations.
\newblock {\em Acta Math.}, 119:147--171, 1967.

\bibitem{hsu}
E.~P. Hsu.
\newblock {\em Stochastic analysis on manifolds}, volume~38 of {\em Graduate
  Studies in Mathematics}.
\newblock American Mathematical Society, Providence, RI, 2002.

\bibitem{iw}
N.~Ikeda and S.~Watanabe.
\newblock {\em Stochastic differential equations and diffusion processes},
  volume~24 of {\em North-Holland Mathematical Library}.
\newblock North-Holland Publishing Co., Amsterdam; Kodansha, Ltd., Tokyo,
  second edition, 1989.

\bibitem{jerison}
D.~S. Jerison and A.~S{\'a}nchez-Calle.
\newblock Estimates for the heat kernel for a sum of squares of vector fields.
\newblock {\em Indiana Univ. Math. J.}, 35(4):835--854, 1986.

\bibitem{kaimanovich}
V.~A. Kaimanovich.
\newblock Discretization of bounded harmonic functions on {R}iemannian
  manifolds and entropy.
\newblock In {\em Potential theory ({N}agoya, 1990)}, pages 213--223. de
  Gruyter, Berlin, 1992.

\bibitem{kn}
S.~Kobayashi and K.~Nomizu.
\newblock {\em Foundations of differential geometry. {V}ol {I}}.
\newblock Interscience Publishers, a division of John Wiley \& Sons, New
  York-London, 1963.

\bibitem{lyonssullivan}
T.~Lyons and D.~Sullivan.
\newblock Function theory, random paths and covering spaces.
\newblock {\em J. Differential Geom.}, 19(2):299--323, 1984.

\bibitem{omori}
H.~Omori.
\newblock On global hypoellipticity of horizontal {L}aplacians on compact
  principal bundles.
\newblock {\em Hokkaido Math. J.}, 20(2):185--194, 1991.

\end{thebibliography}
\end{document}